\documentclass[11pt]{article}

\usepackage[latin1]{inputenc}

\usepackage{amsfonts,amsmath,amssymb,amsthm,graphicx,epsfig,float}
\usepackage[T1]{fontenc}

\usepackage[english,francais]{babel}

{
  \newtheorem{theoreme}{Th\'eor\`eme}
  \newtheorem*{theoreme*}{Th\'eor\`eme}
  \newtheorem{lemme}[theoreme]{Lemme}

  \newtheorem{proposition}[theoreme]{Proposition}
\newtheorem*{corollaire*}{Corollaire}
\newtheorem*{proposition*}{Proposition}
\theoremstyle{remark}
  \newtheorem*{remarque*}{Remarque}
}

\newcounter{ex}

\newenvironment{rem*}{
  \noindent\textbf{Remarque. }}{}

\newcommand{\Cc}{\mathbb{C}}
\newcommand{\Nn}{\mathbb{N}}

\newcommand{\Zz}{\mathbb{Z}}
\newcommand{\Pp}{\mathbb{P}}

\title{{\bf Endomorphismes aléatoires dans les espaces projectifs II}}
\author{Henry de Thélin}
\date{}

\begin{document}
\maketitle

%$\Cc, \Rr$%

%% Redefinition Titre
\def\figurename{{Fig.}}%
\def\proofname{Preuve}% for AMS-\LaTeX
\def\contentsname{Sommaire}%
%% Fin

\begin{abstract}

Nous étudions des suites aléatoires d'endomorphismes holomorphes de $\Pp^k(\Cc)$.

\end{abstract}

\selectlanguage{english}
\begin{center}
{\bf{ }}
\end{center}

\begin{abstract}

We study random holomorphic endomorphisms of $\Pp^k(\Cc)$.

\end{abstract}

\selectlanguage{francais}

Mots-clefs: dynamique complexe, applications aléatoires, entropie.

Classification: 32U40, 32H50.

\section*{{\bf Introduction}}
\par

A partir d'un endomorphisme holomorphe de $\Pp^k(\Cc)$, $f$, de degré
$d \geq 2$,  Forn{\ae}ss et Sibony ont défini le courant de
Green $T$ associé à $f$ (voir \cite{FS} et \cite{FS1}), dont le
support est l'ensemble de Julia de $f$. Si $\omega$ désigne la forme de Fubini-Study de $\Pp^k(\Cc)$, le courant de Green est obtenu comme limite au sens des courants de la suite $\frac{(f^n)^* \omega}{d^n}$.

Ce courant possède un potentiel continu: on peut donc définir son auto-intersection $\mu= T^k$ (voir \cite{FS}). La mesure $\mu$ ainsi obtenue est l'unique
mesure d'entropie maximale $k \log(d)$ (voir \cite{BD2}) et elle a ses
exposants de Lyapunov minorés par $\frac{\log(d)}{2}$ (voir
\cite{BD1}). Par ailleurs, $\mu$ est la limite de la suite de probabilités $\frac{(f^n)^* \omega^k}{d^{kn}}$.

Les convergences des suites $\frac{(f^n)^* \omega}{d^n}$ et $\frac{(f^n)^* \omega^k}{d^{kn}}$ ont été généralisées dans plusieurs directions. L'une d'entre elle consiste à remplacer $f^n$ par $f_n \circ \dots \circ f_0$ où les $f_n$ sont soit des endomorphismes holomorphes aléatoires proches d'un endomorphisme holomorphe $f$ (voir \cite{FS2} et \cite{FW}), soit, dans le cas des mesures, des applications méromorphes qui vérifient certaines propriétés (voir \cite{DS2}). Dans l'article précédent (voir \cite{Det}), nous avons généralisé ce type de résultat. Précisons tout d'abord les théorèmes que nous avons obtenus.

L'ensemble des applications rationnelles de degré $d$ de $\Pp^k(\Cc)$ forme un espace projectif $\Pp^N(\Cc)$ où $N=(k+1) \frac{(d+k)!}{d!k!} -1$. Dans cet espace $\Pp^N(\Cc)$, on notera $\mathcal{H}_d$ les points qui correspondent à des endomorphismes holomorphes de degré $d$ de $\Pp^k(\Cc)$ et $\mathcal{M}$ le complémentaire de $\mathcal{H}_d$ dans $\Pp^N(\Cc)$.

 Considérons $F$ une application mesurable de $\Pp^N(\Cc)$ dans $\Pp^N(\Cc)$ et $\Lambda$ une mesure ergodique et invariante par $F$ (par exemple $F$ un endomorphisme holomorphe de $\Pp^N(\Cc)$ et $\Lambda$ sa mesure de Green). Si $f_0$ est un point de $\Pp^N(\Cc)$ (que l'on prendra générique pour $\Lambda$), on peut considérer la suite $f_n=F^n(f_0)$. Cela donne une suite d'applications rationnelles qui suit en quelque sorte une loi dictée par $\Lambda$.

Pour $f_0 \in \Pp^N(\Cc)$ un endomorphisme holomorphe de $\Pp^k(\Cc)$, on notera $f_i= F^{i} (f_0)$ ($i \in \Nn$) et $F_n$ la composée $F_n= f_n \circ \cdots \circ f_0$ (pour $n \in \Nn$). On a montré dans \cite{Det} le

\begin{theoreme*}

On suppose que 

$$\int \log dist(f, \mathcal{M}) d \Lambda(f) > - \infty.$$

Alors il existe un ensemble $A$ de mesure pleine pour $\Lambda$ tel que pour tout endomorphisme holomorphe $f_0$ de $\Pp^k(\Cc)$ avec $f_0 \in A$, on ait $\frac{F_n^* \omega}{d^{n+1}}$ qui converge vers un courant $T(f_0)$. Ce courant est appelé courant de Green aléatoire (associé à $f_0$).

\end{theoreme*}

Dans \cite{Det}, on a vu que le courant de Green aléatoire ci-dessus est à potentiel continu: on peut donc définir son auto-intersection $T(f_0)^l$ pour $l$ compris entre $1$ et $k$. Pour $l=k$, on appelle $\mu(f_0)=T(f_0)^k$ mesure de Green aléatoire (associée à $f_0$). 

Rappelons (voir \cite{Det}) que l'ensemble $A$ ci-dessus est l'ensemble des bons points pour le théorème de Birkhoff pour la mesure $\Lambda$ et la fonction intégrable $\log dist(f, \mathcal{M})$. Il vérifie $F(A) \subset A$. En particulier, dès que $f_0$ est dans $A$, on peut définir les courants $T(f_i)^l$ pour $i \in \Nn$. Ces courants ont des propriétés d'invariance: on a $d^{-l} f_i^* T(f_{i+1})^l=T(f_i)^l$ et $(f_i)_* T(f_i)^l= d^{k-l} T(f_{i+1})^l$. Grâce à ces invariances, nous avons obtenu dans \cite{Det} un théorème de mélange aléatoire: 

\begin{theoreme}{\label{melange}}

On considère une suite $(f_n)_{n \in \Nn}$ d'endomorphismes holomorphes de degrés $d \geq 2$ et une suite de probabilités $(\mu(f_n))_{n \in \Nn}$ telle que pour tout $n \in \Nn$ on ait $f_n^*(\mu(f_{n+1}))=d^k \mu(f_{n})$ et $\mu(f_n)=(\omega + dd^c g_n)^k$ avec $g_n$ des fonctions continues.

Alors pour $\varphi \in L^{\infty}(\Pp^k)$ et $\psi \in DSH(\Pp^k)$, on a

$$| \langle \mu(f_0), (f_{n-1} \circ \cdots \circ f_0)^* \varphi \psi \rangle - \langle \mu(f_n) , \varphi \rangle \langle \mu(f_0), \psi \rangle | \leq C d^{-n} (1+ \| g_n \|_{\infty}  )^2 \|  \varphi \|_{\infty} \| \psi \|_{\mbox{DSH}}.$$

Ici $C$ est une constante qui ne dépend que de $\Pp^k$.

\end{theoreme}

Classiquement, les applications aléatoires peuvent se voir d'une autre façon: avec un produit semi-direct. C'est ce que nous allons faire dans cet article. Les produits semi-directs ont été étudiés dans \cite{Jo1} et \cite{Jo2} dans le cas où les fibres sont de dimension $1$. Ici on considère des fibres de dimension quelconque. Soit $X= \Pp^N(\Cc) \times \Pp^k(\Cc)$ et $\tau:X \longrightarrow X$ définie par $\tau(f,x)=(F(f), f(x))$. Pour $f \in A$, on a des mesures $\mu(f)$ et elles ont comme propriété d'invariance $f_* \mu(f)= \mu(F(f))$ et $f^* \mu(F(f))=d^k \mu(f)$.

Sur l'espace $X$, on peut définir (voir le paragraphe \ref{mesure}) une mesure $\alpha$ telle que pour tout borélien $B$ de $X$

$$\alpha(B):= \int \mu(f)(B \cap \{f\} \times \Pp^k(\Cc)) d \Lambda(f)$$

où l'on identifie $\{f\} \times \Pp^k(\Cc)$ avec $\Pp^k(\Cc)$.

On obtient alors

\begin{proposition} On suppose que 

$$\int \log dist(f, \mathcal{M}) d \Lambda(f) > - \infty.$$

La mesure $\alpha$ ci-dessus est bien définie, elle est invariante par $\tau$ et ergodique. Si $\Lambda$ est mélangeante, $\alpha$ l'est aussi.

\end{proposition}

A partir d'une mesure $\beta$ invariante par l'application $\tau$, Abramov et Rohlin ont défini une notion d'entropie mixée $h_{\beta, \mbox{mix}}$ (voir  \cite{AR} et \cite{LW}) ). Cette entropie vérifie une inégalité: si $h_{\mbox{top}}(\Lambda)$ désigne l'entropie topologique aléatoire associée à $\Lambda$ (voir \cite{Ki}, \cite{LW} et le paragraphe \ref{topologique}), on a $h_{\beta, \mbox{mix}} \leq h_{\mbox{top}}(\Lambda)$ par \cite{Ki} et \cite{LW}. Dans cet article, nous montrons que $h_{\mbox{top}}(\Lambda)$ est majorée par $k \log d$ et que l'entropie mixée de $\alpha$ est minorée par $k \log d$. On aura alors

\begin{theoreme} On suppose que $\int \log dist(f, \mathcal{M}) d \Lambda(f) > - \infty.$

Alors

$$h_{\alpha, \mbox{mix}}=h_{\mbox{top}}(\Lambda)= k \log d.$$

\end{theoreme}

La mesure $\alpha$ est donc d'entropie mixée maximale.

Dans un dernier paragraphe, nous montrerons un théorème d'hyperbolicité pour les mesures $\mu(f)$ comme dans \cite{Det1}. Ce théorème donnera une généralisation du résultat de Briend et Duval (voir \cite{BD1}): moralement, lorsque l'on part de $(f_1,x)$ générique pour $\alpha$, on verra que $\| D (f_n \circ \cdots \circ f_1)(x)v \| \gtrsim d^{n/2}$ pour tout vecteur unitaire $v$. Plus exactement, il s'agira de montrer que les exposants de Lyapounov de $\alpha$ associé à un certain cocycle sont minorés par $\frac{\log d}{2}$ quand on a $\int \log dist(f, \mathcal{M}) d \Lambda(f) > - \infty$.

\bigskip

{\bf Remerciements:} C'est avec plaisir que je remercie Jerôme Buzzi pour les discussions que nous avons eues sur le théorème d'Oseledets dans le cas non-intégrable.

\section{\bf Propriétés de la mesure $\alpha$}{\label{mesure}}

Dans ce paragraphe, nous montrons que $\alpha$ est bien définie, invariante par $\tau$ et ergodique. Nous montrerons aussi qu'elle est mélangeante lorsque $\Lambda$ l'est.

\subsection{\bf Définition de $\alpha$}

Rappelons que l'on note $A$ l'ensemble des bons points du théorème de Birkhoff pour la mesure $\Lambda$ et la fonction intégrable $\log dist(f, \mathcal{M})$.

Pour montrer que la mesure $\alpha$ est bien définie, il suffit de voir que pour $B$ borélien de $X$, $f \rightarrow \chi_A(f) \mu(f)(B \cap \{f \} \times \Pp^k)$ est mesurable pour $\Lambda$ ce qui revient à montrer que pour $\varphi$ fonction continue sur $X$ la fonction $f \rightarrow \chi_A(f) \int \varphi(f,x) \mu(f)(x)$ est mesurable pour $\Lambda$.

On a vu, dans l'article précédent (\cite{Det}), que pour $f_0 \in A$, $\mu(f_0)$ est limite de $\mu_n(f_0)=\frac{F_n^* \omega^k}{d^{(n+1)k}}$ où $F_n= f_n \circ \cdots \circ f_0$ et $f_i=F^{i}(f_0)$. Comme $f_0 \in A$, les endomorphismes $f_i$ ne sont pas dans $\mathcal{M}$ (car $F(A) \subset A$ et $A \cap \mathcal{M} = \emptyset$). Les coefficients de la forme lisse $\mu_n(f_0)$ évalués en $x$ dépendent donc de façon mesurable de $f_0$ et $x$. En particulier $f \rightarrow  \chi_A(f) \int \varphi(f,x) d \mu_n(f)(x)$ est mesurable pour $\Lambda$ et par passage à la limite $f \rightarrow \chi_A(f) \int \varphi(f,x) \mu(f)(x)$ aussi.

\subsection{\bf La mesure $\alpha$ est invariante par $\tau$}

Classiquement, c'est la relation $f_* \mu(f)= \mu(F(f))$ qui donne l'invariance de $\alpha$. En effet, si $\varphi$ est une fonction continue sur $X$, on a 

$$\int \varphi \circ \tau d \alpha= \int \int \varphi \circ \tau(f,x) d \mu(f)(x) d \Lambda(f)=\int \int \varphi(F(f),f(x)) d \mu(f)(x) d \Lambda(f).$$

Comme on a $f_* \mu(f)= \mu(F(f))$,

$$\int \varphi \circ \tau d \alpha=\int \int \varphi(F(f),x) d \mu(F(f))(x) d \Lambda(f).$$

Maintenant, si on pose $\psi(f)=\int \varphi(f,x) d \mu(f)(x)$,

$$\int \varphi \circ \tau d \alpha= \int \psi(F(f)) d \Lambda(f)= \int \psi(f) d \Lambda(f)$$

par invariance de $\Lambda$ par $F$. Finalement,

$$\int \varphi \circ \tau d \alpha= \int \int \varphi(f,x) d \mu(f)(x) d \Lambda(f)= \int \varphi d \alpha$$

ce qui signifie que $\alpha$ est invariante par $\tau$.

\subsection{\bf Ergodicité et mélange pour $\alpha$}

Montrons d'abord que le fait que $\Lambda$ est ergodique implique que $\alpha$ l'est aussi. La démonstration va reposer sur le théorème de mélange aléatoire (voir le théorème $2$ dans \cite{Det}).

Soient $\varphi$ et $\psi$ des fonctions $C^{\infty}$ sur $X$. Il s'agit de montrer que 

$$\frac{1}{n} \sum_{i=0}^{n-1} \int \varphi(\tau^{i}(f,x)) \psi(f,x) d \alpha(f,x) \longrightarrow \int \varphi d \alpha \int \psi d \alpha.$$

Soit $a_n(f)=\frac{1}{n} \sum_{i=0}^{n-1} \int \varphi(\tau^{i}(f,x)) \psi(f,x) d \mu(f)(x)$ pour $f \in A$.

On a 

$$\frac{1}{n} \sum_{i=0}^{n-1} \int \varphi(\tau^{i}(f,x)) \psi(f,x) d \alpha(f,x)= \int a_n(f) d \Lambda(f).$$

Si $f_0 \in A$, on a (on note $f_i=F^{i}(f_0)$)

$$a_n(f_0)= \frac{1}{n} \sum_{i=0}^{n-1} \int \varphi(F^{i}(f_0), f_{i-1} \circ \cdots \circ f_0(x)) \psi(f_0,x) d \mu(f_0)(x)$$

ce qui s'écrit, en posant $h_i(x)= \varphi(F^{i}(f_0), x)$

$$a_n(f_0)= \frac{1}{n} \sum_{i=0}^{n-1} \int (f_{i-1} \circ \cdots \circ f_0)^* h_i(x) \psi(f_0,x) d \mu(f_0)(x).$$

Utilisons maintenant le théorème $2$ de \cite{Det}. On a (en notant $\psi$ au lieu de $x \rightarrow \psi(f_0,x)$)

$$| \langle \mu(f_0), (f_{i-1} \circ \cdots \circ f_0)^* h_i \psi \rangle - \langle \mu(f_i) , h_i \rangle \langle \mu(f_0), \psi \rangle | \leq C d^{-i} (1+ \| g_i \|_{\infty}  )^2 \|  h_i \|_{\infty} \| \psi \|_{\mbox{DSH}}$$

c'est-à-dire

$$\left| a_n(f_0) - \frac{1}{n} \sum_{i=0}^{n-1} \langle \mu(f_i) , h_i \rangle \langle \mu(f_0), \psi \rangle \right| \leq \frac{C}{n} \sum_{i=0}^{n-1}  d^{-i} (1+ \| g_i \|_{\infty}  )^2 \|  h_i \|_{\infty} \| \psi \|_{\mbox{DSH}}$$

avec $C$ qui ne dépend que de $\Pp^k(\Cc)$. La norme sup $\|  h_i \|_{\infty} $ est plus petite que $\| \varphi \|_{\infty}$, il reste donc à contrôler $g_i$.

Soit $\epsilon > 0$. On a démontré dans le lemme $19$ de \cite{Det} qu'il existe $n_0 \in \Nn$ avec $ \| g_n \|_{\infty} \leq e^{\epsilon n}$ pour $n \geq n_0$. Pour $n \geq n_0$, on a alors

\begin{equation*}
\begin{split}
&\left| a_n(f_0) - \frac{1}{n} \sum_{i=0}^{n-1} \langle \mu(f_i) , h_i \rangle \langle \mu(f_0), \psi \rangle \right| \\
& \leq 
\frac{C}{n} \sum_{i=0}^{n_0-1}  d^{-i} (1+ \| g_i \|_{\infty}  )^2 \| \varphi \|_{\infty} \| \psi \|_{\mbox{DSH}}+
\frac{C}{n} \sum_{i=n_0}^{n-1}  d^{-i} (e^{2 \epsilon i})^2 \| \varphi \|_{\infty} \| \psi \|_{\mbox{DSH}} \\
\end{split}
\end{equation*}

qui converge vers $0$ quand $n$ tend vers l'infini.

On vient de montrer que pour $f_0 \in A$, $\alpha_n(f_0)=a_n(f_0) - \frac{1}{n} \sum_{i=0}^{n-1} \langle \mu(f_i) , h_i \rangle \langle \mu(f_0), \psi \rangle$  converge vers $0$. 

Par ailleurs, comme les $\mu(f_i)$ sont des probabilités, on a $| \alpha_n(f_0) | \leq 2 \| \varphi \|_{\infty} \| \psi \|_{\infty} $ et alors 

$$\int \alpha_n(f_0) d \Lambda(f_0)= \int a_n(f_0) d \Lambda(f_0)   - \frac{1}{n} \sum_{i=0}^{n-1} \int \langle \mu(f_i) , h_i \rangle \langle \mu(f_0), \psi \rangle  d \Lambda(f_0)$$

 tend vers $0$.

Maintenant, si on note $\alpha(f)=\int \varphi(f,x) d \mu(f)(x)$ et $\beta(f)= \int \psi(f,x) d \mu(f)(x)$, on a 

$$\frac{1}{n} \sum_{i=0}^{n-1} \int \langle \mu(f_i) , h_i \rangle \langle \mu(f_0), \psi \rangle  d \Lambda(f_0)= \frac{1}{n} \sum_{i=0}^{n-1} \int \alpha(F^{i}(f)) \beta(f) d \Lambda(f) $$

qui converge vers 

$$\int \alpha(f) d \Lambda(f) \int \beta(f) d \Lambda(f)$$

car $\Lambda$ est ergodique.

Finalement, on a montré que 

$$\int a_n(f) d \Lambda(f)=\frac{1}{n} \sum_{i=0}^{n-1} \int \varphi(\tau^{i}(f,x)) \psi(f,x) d \alpha(f,x)$$

converge vers 

$$\int \alpha(f) d \Lambda(f) \int \beta(f) d \Lambda(f)=\int \varphi d \alpha \int \psi d \alpha.$$

C'est ce que l'on voulait démontrer.

Pour montrer que $\alpha$ est mélangeante lorsque $\Lambda$ l'est, il suffit de montrer que

$$\int \varphi(\tau^{n}(f,x)) \psi(f,x) d \alpha(f,x) \longrightarrow \int \varphi d \alpha \int \psi d \alpha$$

pour $\varphi$ et $\psi$ des fonctions $C^{\infty}$.

En reprenant la preuve précédente en  enlevant tous les $\frac{1}{n} \sum_{i=0}^{n-1}$, on voit que cela revient au fait que (on garde les mêmes notations)

$$ \int \alpha(F^{n}(f)) \beta(f) d \Lambda(f) \longrightarrow \int \alpha(f) d \Lambda(f) \int \beta(f) d \Lambda(f).$$

Mais cette convergence est vraie car $\Lambda$ est mélangeante.

\section{\bf{Entropie}}

Dans un premier paragraphe, nous allons rappeler la définition d'entropie mixée (ou relative) introduite par Abramov et Rohlin (voir \cite{AR} et \cite{LW}) ainsi que la formule qui relie l'entropie métrique de $\alpha$, celle de $\Lambda$ et l'entropie mixée. Dans le second, nous définirons l'entropie topologique aléatoire $h_{\mbox{top}}(\Lambda)$ et nous en donnerons une majoration par $k \log d$. Dans le troisième paragraphe, nous montrerons que l'entropie mixée de $\alpha$ est maximale et vaut $k \log d$. Cela généralise certains résultats de Jonsson au cas où les fibres du produit semi-direct ont une dimension plus grande que $1$ (voir \cite{Jo2}).

\subsection{{\bf Entropie mixée}}{\label{mixee}}

Commençons par rappeler la définition d'entropie mixée associée à la mesure $\alpha$. Cette entropie a été introduite par Abramov et Rohlin dans \cite{AR} dans le cas où $\alpha$ est un produit de mesure. Cette définition a ensuite été étendue au cadre qui nous concerne par Ledrappier et Walters (voir \cite{LW} et aussi \cite{Bo} pour la définition que l'on va prendre). 

On a tout d'abord (voir \cite{Bo} théorème 2.2)

\begin{proposition}

Soit $\xi$ une partition finie de $\Pp^k(\Cc)$. Pour $\Lambda$ presque tout $f_1$ la limite suivante 

$$\lim_{n \rightarrow \infty} \frac{1}{n} H_{\mu(f_1)} \left( \bigvee_{i=0}^{n-1} (f_i \circ \cdots \circ f_1)^{-1}(\xi) \right)$$

existe et est constante. Sa valeur est notée $h_{\alpha, \mbox{mix}}(\xi)$.

Ici par convention $(f_i \circ \cdots \circ f_1)^{-1}$ vaut l'identité si $i=0$.

\end{proposition}

On définit alors l'entropie mixée par 

$$h_{\alpha, \mbox{mix}}= \sup h_{\alpha, \mbox{mix}}(\xi)$$

où le sup est pris sur l'ensemble des partitions finies $\xi$ de $\Pp^k(\Cc)$.

On a maintenant trois quantités: les entropies métriques $h_{\alpha}(\tau)$ et $h_{\Lambda}(F)$ ainsi que l'entropie mixée $h_{\alpha, \mbox{mix}}$. Si on note $\pi$ la projection de $\Pp^N(\Cc) \times \Pp^k(\Cc)$ sur $\Pp^N(\Cc)$, on a $\pi_* \alpha= \Lambda$ et les quantités précédentes sont donc reliées par une formule qui est due à Abramov et Rohlin (voir \cite{AR}):

\begin{theoreme} (Abramov-Rohlin \cite{AR})

On a $h_{\alpha}(\tau)=h_{\Lambda}(F) + h_{\alpha, \mbox{mix}}$.

\end{theoreme}

Dans \cite{AR}, le théorème ci-dessus est donné dans le cas où $\alpha$ est le produit de deux mesures. Cela a été généralisé à notre cadre par Ledrappier-Walters (voir \cite{LW}) et Bogenschütz-Crauel (voir \cite{BC}).

\subsection{{\bf Entropie topologique}}{\label{topologique}}

Commençons par rappeler la définition d'entropie topologique dans le cadre des applications aléatoires (voir \cite{Ki} p.67 et suivantes).

Si $\beta$ est un recouvrement de $\Pp^k(\Cc)$ par des ouverts, on notera $N(\beta)$ le nombre minimal d'ensemble de $\beta$ pour recouvrir $\Pp^k(\Cc)$ (c'est un nombre fini car $\Pp^k(\Cc)$ est compact). On pose $\mathcal{H}(\beta)= \log N(\beta)$. La proposition suivante permet de définir l'entropie topologique:

\begin{proposition}

Il existe une constante $h(\beta)$ telle que pour $\Lambda$-presque tout $f_1$ la limite 
$$\lim_{n \rightarrow \infty} \frac{1}{n} \mathcal{H} \left(\bigvee_{i=0}^{n-1} (f_i \circ \cdots \circ f_1)^{-1}(\beta) \right)$$
existe et vaut $h(\beta)$. Ici on prend encore la convention que $f_i \circ \cdots \circ f_1$ vaut l'identité si $i=0$.

\end{proposition}

\begin{proof}

On commence par prendre $f_1$ dans l'ensemble $A$ de sorte que tous les $f_i$ soient des endomorphismes holomorphes.

Maintenant, en suivant exactement la preuve de \cite{Ki} p.69 et en posant

$$b_n(f_1)=\mathcal{H} \left( \bigvee_{i=0}^{n-1} (f_i \circ \cdots \circ f_1)^{-1}(\beta) \right)=\mathcal{H} \left( \bigvee_{i=0}^{n-1} (F^{i-1}(f_1) \circ \cdots \circ f_1)^{-1}(\beta) \right),$$

on a

$$b_{n+m}(f_1) \leq b_n(f_1) + b_m(F^n(f_1)).$$

On conclut alors en utilisant le théorème sous-additif de Kingman car $\Lambda$ est ergodique.

\end{proof}

L'entropie topologique aléatoire se définit par

$$h_{\mbox{top}}(\Lambda)=\sup h(\beta).$$

Remarquons que c'est une quantité qui dépend de $F$ et $\Lambda$ que l'on a choisis au départ.

Comme dans \cite{Ki}, on peut définir cette entropie d'une autre façon.

Pour $f_1$ dans l'ensemble $A$, on dira qu'un ensemble $E$ est $(n , \epsilon, f_1)$-séparé si pour tous $x$ et $y$ dans $E$ avec $x \neq y$, on a

$$d_{f_1}^n(x,y):=\max_{i=0, \cdots , n-1} dist( f_i \circ \cdots \circ f_1(x), f_i \circ \cdots \circ f_1(y)) \geq \epsilon.$$

On note $s(n, \epsilon, f_1)$ le cardinal maximal d'un ensemble $(n , \epsilon, f_1)$-séparé dans $\Pp^k(\Cc)$.

Alors 

\begin{proposition}

Pour $\Lambda$ presque tout $f_1$, on a 

$$h_{\mbox{top}}(\Lambda)=\lim_{\epsilon \rightarrow 0} \limsup_{n \rightarrow + \infty} \frac{1}{n} \log s(n, \epsilon, f_1)=\lim_{\epsilon \rightarrow 0} \liminf_{n \rightarrow + \infty} \frac{1}{n} \log s(n, \epsilon, f_1).$$

\end{proposition}

La preuve est faîte dans \cite{Ki} p.72-74. 

Il y a une inégalité entre l'entropie mixée et l'entropie topologique définie ci-dessus. Elle résulte d'une adaptation facile de \cite{Ki} p. 78 (voir aussi \cite{LW} pour un résultat plus complet dans le cas continu). En effet, on a

\begin{theoreme} 

$$h_{\alpha, \mbox{mix}} \leq h_{\mbox{top}}(\Lambda).$$

\end{theoreme}

Nous verrons dans la suite que $\alpha$ est une mesure d'entropie mixée maximale, c'est-à-dire que l'inégalité du théorème ci-dessus est en fait une égalité.

En utilisant des arguments de Gromov (voir \cite{Gr}), on va d'abord majorer l'entropie topologique $h_{\mbox{top}}(\Lambda)$. On a 

\begin{theoreme}{\label{Gromov}}

$$h_{\mbox{top}}(\Lambda) \leq k \log d.$$

\end{theoreme}

\begin{proof}

On considère $f_1 \in A$ qui vérifie la conclusion de la  proposition précédente.

Soit $\{x_1, \cdots , x_N\}$ un ensemble $(n , \epsilon, f_1)$-séparé dans $\Pp^k(\Cc)$ de cardinal maximal. Pour $i=1 , \cdots , N$ , les $x_i$ donnent des points

$$X_i=(x_i, f_1(x_i), f_2 \circ f_1(x_i), \cdots , f_{n-1} \circ \cdots \circ f_1(x_i))$$

qui sont $\epsilon$-séparés dans $(\Pp^k(\Cc))^n$. Les boules $B(X_i, \frac{\epsilon}{2})$ sont donc disjointes. Si on note

$$\Gamma_n = \{ (x,f_1(x), f_2 \circ f_1(x), \cdots , f_{n-1} \circ \cdots \circ f_1(x)), x \in \Pp^k(\Cc) \}$$

le théorème de Lelong implique que le volume de $\Gamma_n$ intersecté avec une boule $B(X_i, \frac{\epsilon}{2})$ est minoré par $C(k) \left( \frac{\epsilon}{2} \right)^{2k}$. Cela donne une minoration du volume de $\Gamma_n$ par  $C(k) \left( \frac{\epsilon}{2} \right)^{2k}N$.

Majorons maintenant ce volume en utilisant la cohomologie des applications $f_i$. 

Soit $\omega_n=\sum_{i=1}^n \pi_i^* \omega$ où $\pi_i$ ($i=1, \cdots , n$) est la projection de $(\Pp^k(\Cc))^n$ sur sa $i$-ème coordonnée et $\omega$ est la forme de Fubini-Study de $\Pp^k(\Cc)$. Le volume de $\Gamma_n$ est égal à

\begin{equation*}
\begin{split}
\int_{\Gamma_n} \omega_n^k &= \sum_{1 \leq n_1 , \cdots , n_k \leq n} \int_{\Gamma_n} \pi_{n_1}^* \omega \wedge \cdots \wedge \pi_{n_k}^* \omega \\
&=\sum_{1 \leq n_1 , \cdots , n_k \leq n} \int_{\Pp^k(\Cc)} (f_{n_1 -1} \circ \cdots \circ f_1)^* \omega \wedge \cdots \wedge (f_{n_k -1} \circ \cdots \circ f_1)^* \omega.\\
\end{split}
\end{equation*}

Comme $f_1$ est dans $A$, les $f_i$ sont des endomorphismes holomorphes de degré $d$. En particulier, $(f_{n_l -1} \circ \cdots \circ f_1)^* \omega$ est cohomologue à $d^{n_l -1} \omega$ et alors

$$\int_{\Gamma_n} \omega_n^k =\sum_{1 \leq n_1 , \cdots , n_k \leq n} d^{n_1 + \cdots + n_k}d^{-k} \leq n^k d^{-k} d^{nk}.$$

On obtient ainsi

$$N \leq \frac{n^k d^{-k} d^{nk}}{C(k) \left( \frac{\epsilon}{2} \right)^{2k}}$$

ce qui donne le théorème.

\end{proof}

\subsection{{\bf Quelques propriétés de l'entropie mixée}}

Dans ce paragraphe, nous donnons quelques propriétés sur l'entropie qui nous seront utiles pour la suite.

Commençons par énoncer un théorème du type Shannon-McMillan-Breiman pour les applications aléatoires. Soit $\xi= \{A_1, \cdots , A_p \}$ une partition finie de $\Pp^k(\Cc)$ et $f_1$ un endomorphisme dans $A$. On note 
$$I_{\mu(f_1)}(\xi)(x)= -\sum_{i=1}^p \chi_{A_i}(x) \log \mu(f_1)(A_i).$$

Par le théorème 4.2 de \cite{Bo}, on a

\begin{theoreme}

Soit $\xi$ une partition finie. Pour $\alpha$ presque tout $(f_1,x)$

$$\lim_{n \rightarrow \infty}  \frac{1}{n} I_{ \mu(f_1)} \left( \bigvee_{i=0}^{n-1} (f_i \circ \cdots \circ f_1)^{-1}(\xi) \right) (x)=h_{\alpha, \mbox{mix}}(\xi).$$

\end{theoreme}

L'entropie mixée de $\alpha$ est plus petite que $k \log d$. C'est donc une quantité finie et on peut aussi appliquer le théorème 2.1 de \cite{Z} qui est un théorème du type Brin-Katok pour les applications aléatoires. On note ici $B_{d_{f_1}^n}(x, \epsilon)$ la boule de centre $x$ et de rayon $\epsilon$ pour la métrique $d_{f_1}^n$ définie précédemment:

\begin{theoreme}

Pour $\alpha$ presque tout $(f_1,x)$ on a:

\begin{equation*}
\begin{split}
h_{\alpha, \mbox{mix}}&= \lim_{\epsilon \rightarrow 0} \liminf_{n \rightarrow + \infty} - \frac{1}{n} \log \mu(f_1)(B_{d_{f_1}^n}(x, \epsilon)) \\
&= \lim_{\epsilon \rightarrow 0} \limsup_{n \rightarrow + \infty} - \frac{1}{n} \log \mu(f_1)(B_{d_{f_1}^n}(x, \epsilon)). \\
\end{split}
\end{equation*}
\end{theoreme}

Nous montrons maintenant que $\alpha$ est d'entropie mixée maximale. On a même mieux car:

\begin{theoreme}

$$h_{\alpha, \mbox{mix}}=h_{\mbox{top}}(\Lambda)= k \log d.$$

\end{theoreme}

\begin{proof}

Pour démontrer ce théorème il suffit de suivre la preuve du théorème 5.2 de \cite{Jo2} en changeant $\log d$ par $k \log d$ et $\mu_x$ par $\mu(f_1)$. En effet, d'une part la mesure $\mu(f_1)$ ne charge pas les sous-ensembles analytiques de $\Pp^k(\Cc)$ car elle est égale à $T(f_1)^k$ où $T(f_1)$ est un $(1,1)$ courant à potentiel continu. D'autre part, on a la relation $f_1^* \mu(f_2)= d^k \mu(f_1)$ pour $f_1 \in A$ (voir \cite{Det}). 

\end{proof}

\section{\bf{Exposants de Lyapounov et hyperbolicité de la mesure $\alpha$}}

Il s'agit ici d'étudier la dynamique de la suite $f_0, \cdots, f_n$ pour $f_0$ générique pour $\Lambda$. Pour cela, nous allons estimer des exposants de Lyapounov associés à un certain cocycle pour la mesure $\alpha$. 

Commençons par décrire ce cocycle.

\subsection{\bf Cocycle et théorème d'Oseledets}

On note toujours $X= \Pp^N(\Cc) \times \Pp^k(\Cc)$ et $\tau:X \longrightarrow X$ définie par $\tau(f,x)=(F(f), f(x))$. 

Soit 
$$\widehat{X}:= \{ \widehat{\beta}=( \cdots, \beta_{-n} , \cdots , \beta_0 , \cdots , \beta_n, \cdots) \in X^{\Zz} \mbox{ , } \tau(\beta_n)= \beta_{n+1} \mbox{   } \forall n \in \Zz \}.$$

Dans cet espace, $\tau$ induit une application $\sigma$ qui est le décalage à gauche et si on note $\pi$ la projection canonique $\pi(\widehat{\beta})=\beta_0$, la probabilité $\alpha$ se relève en une probabilité $\widehat{\alpha}$ invariante par $\sigma$, ergodique et qui vérifie $\pi_* \widehat{\alpha}= \alpha$.

Si on considère $\mathcal{I}= \{ \beta=(f,x) \in X \mbox{  avec  } x \in I(f) \}$ où $I(f)$ est l'ensemble d'indétermination de $f$, on a

$$\alpha(\mathcal{I})= \int \mu(f)(\mathcal{I} \cap \{f\} \times \Pp^k(\Cc) ) d \Lambda(f)=\int \mu(f)(I(f)) d \Lambda(f)=0$$

car par hypothèse $\int \log dist(f, \mathcal{M}) d \Lambda(f) > - \infty$, ce qui signifie que $I(f)$ est vide pour $\Lambda$ presque tout $f$.

En particulier, si on pose

$$\widehat{X}^*:= \{ \widehat{\beta} \in \widehat{X} \mbox{ , } \beta_n \notin \mathcal{I} \mbox{ , } \forall n \in \Zz \},$$

cet ensemble est invariant par $\sigma$ et on a $\widehat{\alpha}(\widehat{X}^*)=1$.

Maintenant, on munit $\Pp^k(\Cc)$ d'une famille de cartes $(\tau_x)_{x \in \Pp^k(\Cc)}$ telles que $\tau_x(0)=x$, $\tau_x$ est définie sur une boule $B(0, \epsilon_0) \subset \Cc^k$ avec $\epsilon_0$ indépendant de $x$ et la norme de la dérivée première et seconde de $\tau_x$ sur $B(0, \epsilon_0)$ est majorée par une constante indépendante de $x$. Pour construire ces cartes il suffit de partir d'une famille finie $(U_i, \psi_i)$ de cartes de $\Pp^k(\Cc)$ et de les composer par des translations.

Dans toute la suite, si $f: \Pp^k(\Cc) \longrightarrow \Pp^k(\Cc)$ est une application rationnelle, on notera $f_x= \tau_{f(x)}^{-1} \circ f \circ \tau_x$ qui est définie au voisinage de $0$ quand $x$ n'est pas dans
$I(f)$.

Le cocycle auquel nous allons appliquer la théorie de Pesin est le suivant:

\begin{equation*}
\begin{split}
A : & \widehat{X}^* \longrightarrow M_k(\Cc)\\
& \widehat{\beta} \longrightarrow Df_x(0)\\
\end{split}
\end{equation*}

où $M_k(\Cc)$ est l'ensemble des matrices carrées $k \times k$ à coefficients dans $\Cc$ et $\pi(\widehat{\beta})=\beta= (f,x)$. Afin d'avoir un théorème du type Oseledets, nous aurons besoin du lemme suivant:

\begin{lemme}

$$\int \log^+ \| A( \widehat{\beta}) \| d \widehat{\alpha} ( \widehat{\beta}) < + \infty.$$

\end{lemme}

\begin{proof}

Pour $\beta=(f,x)$, on pose $h(\beta)= \log^+ \|Df_x(0) \|$. On a alors si $\pi(\widehat{\beta})=\beta$,

\begin{equation*}
\begin{split}
&\int \log^+ \| A( \widehat{\beta}) \| d \widehat{\alpha} ( \widehat{\beta}) = \int  h(\beta) d \widehat{\alpha} ( \widehat{\beta})\\
&=\int  h \circ \pi ( \widehat{\beta}) d \widehat{\alpha} ( \widehat{\beta})=\int  h( \beta) d \alpha(\beta)\\
&=\int \int \log^+ \|Df_x(0) \| d \mu(f)(x) d \Lambda(f).
\end{split}
\end{equation*}

Maintenant, 

$$ \|Df_x(0) \| \leq C \|Df(x) \| \leq C' dist(f, \mathcal{M})^{-p}$$

(voir la démonstration de la proposition $3$ de \cite{Det}). Le lemme découle donc de l'intégrabilité de la fonction $\log dist(f, \mathcal{M})$ pour la mesure $\Lambda$.

\end{proof}

Grâce à ce lemme, on obtient un théorème du type Oseledets (voir  \cite{FLQ} et \cite{Th} ainsi que le théorème 2.3 de \cite{N}, le théorème 6.1 dans \cite{Du} et \cite{M}):

\begin{theoreme}{\label{pesin}}

 Il existe des réels $\lambda_1 > \lambda_2 > \cdots > \lambda_l \geq - \infty$, des entiers $m_1, \cdots, m_l$ et un ensemble $\widehat{\Gamma}$ de mesure pleine pour $\widehat{\alpha}$ tels que pour $\widehat{\beta} \in \widehat{\Gamma}$ on ait une décomposition de $\Cc^k$ de la forme $\Cc^k= \bigoplus_{i=1}^{l} E_i(\widehat{\beta})$ où les $E_i(\widehat{\beta})$ sont des sous-espaces vectoriels de dimension $m_i$ qui vérifient:

1) $A(\widehat{\beta}) E_i(\widehat{\beta}) \subset E_i(\sigma(\widehat{\beta}))$ avec égalité si $\lambda_i > - \infty$.

2) Pour $v \in E_i(\widehat{\beta}) \setminus \{0 \}$, on a 
$$\lim_{n \rightarrow + \infty} \frac{1}{n} \log \|A(\sigma^{n-1}(\widehat{\beta})) \cdots A(\widehat{\beta}) \| = \lambda_i.$$

Si de plus, $\lambda_i > - \infty$, on a la même limite quand $n$ tend vers $- \infty$.

Pour tout $\gamma > 0$, il existe une fonction $C_{\gamma} :  \widehat{\Gamma} \longrightarrow GL_k(\Cc)$ telle que pour $\widehat{\beta} \in \widehat{\Gamma}$:

1) $\lim_{n \rightarrow \infty} \frac{1}{n} \log \| C^{\pm 1}_{\gamma} (\sigma^n(\widehat{\beta})) \|=0$ (on parle de fonction tempérée).

2) $C_{\gamma}(\widehat{\beta})$ envoie la décomposition standard $\bigoplus_{i=1}^{l} \Cc^{m_i}$ sur $\bigoplus_{i=1}^{l} E_i(\widehat{\beta})$.

3) La matrice $A_{\gamma}(\widehat{\beta})= C_{\gamma}^{-1}(\sigma(\widehat{\beta})) A(\widehat{\beta}) C_{\gamma}(\widehat{\beta})$ est diagonale par bloc $(A^1_{\gamma}(\widehat{\beta}), \cdots, A^l_{\gamma}(\widehat{\beta}))$ où chaque $A^{i}_{\gamma}(\widehat{\beta})$ est une matrice carrée $m_i \times m_i$ et
$$\forall v \in \Cc^{m_i}   \mbox{    } \mbox{  on a   } \mbox{    } e^{\lambda_i - \gamma} \|v\| \leq \| A^{i}_{\gamma}(\widehat{\beta}) v \| \leq e^{\lambda_i + \gamma} \|v\|$$

si $\lambda_i > - \infty$ et  

$$\forall v \in \Cc^{m_l}   \mbox{    } \mbox{     } \mbox{    }  \| A^{l}_{\gamma}(\widehat{\beta}) v \| \leq e^{\gamma} \|v\|$$

si $\lambda_l= - \infty$.

\end{theoreme}

\begin{remarque*}

La dernière estimée est donnée sous une forme que l'on utilisera mais on peut l'améliorer: en fait si on fixe $\gamma>0$ et $B>0$, il existe une fonction $C_{\gamma,B} :  \widehat{\Gamma} \longrightarrow GL_k(\Cc)$ qui vérifie les propriétés ci-dessus et pour la dernière, si on note $A_{\gamma,B}(\widehat{\beta})= C_{\gamma,B}^{-1}(\sigma(\widehat{\beta})) A(\widehat{\beta}) C_{\gamma,B}(\widehat{\beta})$, on a $\forall v \in \Cc^{m_l}   \mbox{    } \mbox{     } \mbox{    }  \| A^{l}_{\gamma,B}(\widehat{\beta}) v \| \leq e^{-B} \|v\|$.

\end{remarque*}

Notons maintenant $g_{\widehat{\beta}}$ la lecture de $f_x$ dans les cartes $C_{\gamma}$ c'est-à-dire $g_{\widehat{\beta}}= C_{\gamma}^{-1}(\sigma(\widehat{\beta}))  \circ f_x \circ C_{\gamma}(\widehat{\beta})$ où $\pi(\widehat{\beta})=\beta=(f,x)$. On considère aussi $C$ et $p \geq 5$ tels que 

$$\|Df(x)\| + \|D^2f(x)\| \leq C dist(f, \mathcal{M})^{-p}$$

(voir la démonstration du lemme $6$ de \cite{Det} et le lemme 2.1 de \cite{DiDu}). Dans l'expression précédente est s'est implicitement placé dans une des cartes $\psi_i$ de $\Pp^k(\Cc)$. Grâce au théorème précédent, on a alors

\begin{proposition}{\label{proposition15}}

Il existe une constante $\epsilon_1$ qui ne dépend que de $\Pp^k(\Cc)$ telle que pour $\widehat{\beta} \in \widehat{\Gamma}$ on ait

1) $g_{\widehat{\beta}}(0)=0$.

2) $D g_{\widehat{\beta}}(0)=A_{\gamma}(\widehat{\beta})$.

3) Si on note $g_{\widehat{\beta}}(w)=D g_{\widehat{\beta}}(0)w + h(w)$, on a 
$$\| Dh(w) \| \leq C \|C_{\gamma}^{-1}(\sigma(\widehat{\beta}))\|  \|C_{\gamma}(\widehat{\beta})\|^2 dist(f, \mathcal{M})^{-p} \|w\|$$
pour $\| w \| \leq \frac{\epsilon_1 dist(f, \mathcal{M})^p}{C \|C_{\gamma}(\widehat{\beta})\|}$.

\end{proposition}

\begin{proof}

Commençons par montrer que $g_{\widehat{\beta}}(w)$ est défini pour $\| w \| \leq \frac{\epsilon_1 dist(f, \mathcal{M})^p}{C \|C_{\gamma}(\widehat{\beta})\|}$.

Par construction des cartes $\tau_x$, on peut trouver $\epsilon_1$  qui ne dépend que de $\Pp^k(\Cc)$ tel que pour $x \in \Pp^k(\Cc)$ les $\tau_x$ sont définis sur $B(0, \epsilon_1)$ et les $\tau_x^{-1}$ sur $B(x, \epsilon_1)$.

Pour $\| w \| \leq \frac{\epsilon_1}{\|C_{\gamma}(\widehat{\beta})\|}$, on a 

\begin{equation*}
\begin{split}
dist(f \circ \tau_x \circ C_{\gamma}(\widehat{\beta}) (w), f(x)) &= dist(f \circ \tau_x \circ C_{\gamma}(\widehat{\beta}) (w), f \circ \tau_x(0)) \\
&\leq C dist(f,\mathcal{M})^{-p} \|C_{\gamma}(\widehat{\beta}) (w)\|
\leq  C dist(f,\mathcal{M})^{-p}   \|C_{\gamma}(\widehat{\beta})\|  \|w\|.\\
\end{split}
\end{equation*}

Le dernier terme est plus petit que $\epsilon_1$ si $\|w\| \leq \frac{\epsilon_1 dist(f, \mathcal{M})^p}{C  \|C_{\gamma}(\widehat{\beta})\|}$, ce qui signifie que $g_{\widehat{\beta}}(w)$ est défini pour de tels $w$.

Passons à la preuve de la proposition.

Le point $1$ est évident et le point $2$ découle du théorème précédent.

Pour le troisième point:

$Dg_{\widehat{\beta}}(w)=D g_{\widehat{\beta}}(0) + Dh(w)$, d'où pour $\| w \| \leq \frac{\epsilon_1 dist(f, \mathcal{M})^p}{C \|C_{\gamma}(\widehat{\beta})\|}$

\begin{equation*}
\begin{split}
\| Dh(w) \|&= \|Dg_{\widehat{\beta}}(w)-D g_{\widehat{\beta}}(0)\|\\
&=\| C_{\gamma}^{-1}(\sigma(\widehat{\beta})) \circ Df_x(C_{\gamma}(\widehat{\beta}) (w)) \circ C_{\gamma}(\widehat{\beta}) - C_{\gamma}^{-1}(\sigma(\widehat{\beta})) \circ Df_x(0) \circ C_{\gamma}(\widehat{\beta})\|\\
&\leq \| C_{\gamma}^{-1}(\sigma(\widehat{\beta}))\|  \|Df_x(C_{\gamma}(\widehat{\beta}) (w)) -Df_x(0) \| \| C_{\gamma}(\widehat{\beta}) \| \\
& \leq  C  \| C_{\gamma}^{-1}(\sigma(\widehat{\beta}))\|\|C_{\gamma}(\widehat{\beta}) \|^2 dist(f, \mathcal{M})^{-p}  \|w\|.
\end{split}
\end{equation*}

C'est ce que l'on voulait démontrer.

\end{proof}

Un dernier théorème que nous utiliserons est la transformée de graphe dans le cas non-inversible (voir \cite{Du} théorème 6.4). Dans celui-ci $B_l(0,R)$ désigne la boule de centre $0$ et de rayon $R$ dans $\Cc^l$.

\begin{theoreme}(transformée de graphe cas non-inversible){\label{graphe}}

Soient $A: \Cc^{k_1} \longrightarrow \Cc^{k_1}$, $B: \Cc^{k_2} \longrightarrow \Cc^{k_2}$ des applications linéaires avec $k=k_1 + k_2$. On suppose $A$ inversible, $\|B\| < \|A^{-1}\|^{-1}$ et on note $\xi=1 - \|B\| \|A^{-1}\| \in ]0,1]$. Soient $0 \leq \xi_0 \leq 1$ et $\delta >0$ tels que:

$$\xi_0(1- \xi)+2 \delta(1 + \xi_0) \|A^{-1}\| \leq 1  \mbox{   }\mbox{  et }$$

$$(\xi_0 \|B\|+ \delta(1 + \xi_0))(\|A^{-1}\|^{-1}-\delta(1 + \xi_0))^{-1} \leq \xi_0.$$

Soit $g: B_k(0,R_0) \longrightarrow B_k(0, R_1)$ holomorphe avec $R_0 \leq R_1$, $g(0)=0$, $Dg(0)=(A,B)$ et $\|Dg(w) - Dg(0) \| \leq \delta$ sur $B_k(0,R_0)$. On a

Si $\phi : B_{k_2}(0, R) \longrightarrow \Cc^{k_1}$ vérifie $\phi(0)=0$ et $Lip( \phi) \leq \xi_0$ pour un certain $R \leq R_0$ alors il existe $\psi:  B_{k_2} \left( 0, \frac{R}{\max(1, \|B\| + 2 \delta)} \right) \longrightarrow \Cc^{k_1}$ avec $Lip (\psi) \leq \xi_0$ et $g(\mbox{graphe}(\psi)) \subset \mbox{graphe}(\phi)$.

\end{theoreme}

Pour obtenir cet énoncé il suffit juste d'adapter un peu la preuve du point $2$ du théorème 6.4 de \cite{Du}.

\begin{remarque*}

Le théorème précédent a aussi un sens quand $k_1=0$. Dans ce cas cela revient à faire des branches inverses comme dans \cite{BD2} mais dans un cas non inversible. En voici un énoncé:

Soit $g: B_k(0,R_0) \longrightarrow B_k(0, R_1)$ holomorphe avec $R_0 \leq R_1$, $g(0)=0$, $Dg(0)=(B)$ et $\|Dg(w) - Dg(0) \| \leq \delta$ sur $B_k(0,R_0)$. Si $R \leq R_0$, on a 

$$g \left( B_k \left( 0, \frac{R}{\max(1, \|B\| + 2 \delta)} \right) \right) \subset B_k(0,R).$$

\end{remarque*}

\subsection{\bf Démonstration du théorème d'hyperbolicité}

Le but de ce paragraphe est de montrer que les exposants de Lyapounov $\lambda_1 > \cdots > \lambda_l$ définis précédemment sont supérieurs ou égaux à $\frac{\log(d)}{2}$. Cela généralise le théorème de Briend et Duval aux endomorphismes holomorphes aléatoires (voir \cite{BD1}).

Nous allons suivre la même méthode que dans \cite{Bu}, \cite{N} et \cite{Det1}.

Choisissons $\gamma > 0$ tel que $\gamma p$ soit très petit devant les différences $\lambda_i - \lambda_{i+1}$ pour $i=1, \cdots, l-1$, devant les valeurs absolues des exposants non nuls et devant $\log d$ (ici le $p$ est le même qu'au paragraphe précédent).

Nous avons montré que pour $\alpha$ presque tout $(f_1,x)$ on a 

$$h_{\alpha, \mbox{mix}}= \lim_{\epsilon \rightarrow 0} \liminf_{n \rightarrow + \infty} - \frac{1}{n} \log \mu(f_1)(B_{d_{f_1}^n}(x, \epsilon)) = k \log d.$$

On va faire quelques uniformisations de cette formule.

Soit $\Lambda_{\epsilon,n}= \{(f_1, x) \in X \mbox{  ,  } \mu(f_1)(B_{d_{f_1}^n}(x, \epsilon)) \leq e^{-kn \log d + \gamma n} \}$.

Si $\epsilon$ est assez petit on  a

\begin{equation*}
\begin{split}
\frac{4}{5} &\leq \alpha \left( \left\{ (f_1,x) \in X \mbox{  ,  } \liminf_{n \rightarrow + \infty} - \frac{1}{n} \log \mu(f_1)(B_{d_{f_1}^n}(x, \epsilon)) \geq k \log d -  \frac{\gamma}{2} \right\} \right)\\
& \leq \alpha(\cup_{n_0} \cap_{n \geq n_0} \Lambda_{\epsilon,n}).
\end{split}
\end{equation*}

En particulier si $n_0$ est grand on  a $ \alpha(\cap_{n \geq n_0} \Lambda_{\epsilon,n}) \geq 3/4$.

Rappelons que l'on note $\widehat{\Gamma}$ l'ensemble des bons points pour la théorie de Pesin de la mesure $\widehat{\alpha}$ (voir le théorème \ref{pesin}). On considère (toujours avec les notations de ce théorème)

$$\widehat{\Gamma}_{\alpha_0}= \left\{ \widehat{\beta} \in \widehat{\Gamma} \mbox{  ,  } \alpha_0 \leq \| C_{\gamma}(\widehat{\beta})^{\pm 1} \| \leq \frac{1}{\alpha_0} \right\}.$$

Si $\alpha_0$ est assez petit, on a $\widehat{\alpha}(\widehat{\Gamma}_{\alpha_0}) \geq 3/4$ d'où

$$\alpha(\pi(\widehat{\Gamma}_{\alpha_0}) \cap  (\cap_{n \geq n_0} \Lambda_{\epsilon,n}))= \int \mu(f_1)(\pi(\widehat{\Gamma}_{\alpha_0}) \cap  (\cap_{n \geq n_0} \Lambda_{\epsilon,n} )\cap \{f_1\}\times \Pp^k) d \Lambda(f_1) \geq \frac{1}{2}.$$

On obtient ainsi l'existence de $f_1$ avec $\mu(f_1)(A_{n_0}) \geq 1/2$ où 

$$A_{n_0}=\{ x \in \Pp^k(\Cc) \mbox{ avec } (f_1,x) \in \pi(\widehat{\Gamma}_{\alpha_0}) \cap  (\cap_{n \geq n_0} \Lambda_{\epsilon,n} ) \}$$

 (car on a continué l'identification entre $\{f_1\}\times \Pp^k$ et $\Pp^k$). On peut aussi supposer que $f_1 \in A$ car $A$ est de mesure pleine pour $\Lambda$.

Les points $x \in A_{n_0}$ vérifient $\mu(f_1)(B_{d_{f_1}^n}(x, \epsilon)) \leq e^{-k n\log d + \gamma n}$ pour tout $n \geq n_0$. On peut donc trouver $x_1, \cdots , x_N$ dans $A_{n_0}$ qui sont $(n, \epsilon, f_1)$ séparés avec $N \geq \frac{1}{2} e^{k n\log d - \gamma n}$. Comme les $x_i$ sont dans $A_{n_0}$, il existe $\widehat{\beta_i} \in \widehat{\Gamma}_{\alpha_0}$ bon point de Pesin avec $\pi(\widehat{\beta_i})=(f_1, x_i)$. On va donc pouvoir appliquer la théorie de Pesin à partir du point $x_i$: en particulier on va construire des variétés stables approchées en chacun de ces points. Ces variétés stables vont être produites grâce à la transformée de graphe (voir le théorème \ref{graphe}).

\subsection{\bf{Construction des variétés stables}}

Soit $\xi_0 >0$ très petit devant $\alpha_0$. Dans toute la suite $n$ est pris grand par rapport à $\gamma$ et $\xi_0$.

Soit $x$ un des $x_i$ précédents et $\widehat{\beta} \in \widehat{\Gamma}_{\alpha_0}$ qui s'envoie sur $(f_1,x)$ par $\pi$. On note $E_1(\widehat{\beta}), \cdots, E_l(\widehat{\beta})$ les sous-espaces vectoriels donnés par la théorie de Pesin dans le théorème \ref{pesin}. Si $\lambda_l \leq 0$, on pose $E_s(\widehat{\beta})= \oplus_{i=m}^{l} E_i(\widehat{\beta})$ et $E_u(\widehat{\beta})= \oplus_{i=1}^{m-1} E_i(\widehat{\beta})$ où $\lambda_1 > \cdots > \lambda_{m-1} > 0 \geq \lambda_{m} > \cdots > \lambda_l$ (éventuellement on a $E_u(\widehat{\beta})= \{0\}$ si $\lambda_1$ est négatif). Si $\lambda_l > 0$ on posera $E_s(\widehat{\beta})=E_l(\widehat{\beta})$ et $E_u(\widehat{\beta})=  \oplus_{i=1}^{l-1} E_i(\widehat{\beta})$ (éventuellement on a $E_u(\widehat{\beta})= \{0\}$ si $l=1$) .

Maintenant, on se place dans le repère

$$C_{\gamma}^{-1}(\sigma^{n-1}(\widehat{\beta})) E_u(\sigma^{n-1}(\widehat{\beta})) \oplus C_{\gamma}^{-1}(\sigma^{n-1}(\widehat{\beta}))E_s(\sigma^{n-1}(\widehat{\beta}))$$

et on part de $\{0\}^{d_u} \times B(0, e^{- 4 \gamma n p})$ où $d_u$ est la dimension de $E_u(\sigma^{n-1}(\widehat{\beta}))$ et $B(0, e^{- 4 \gamma n p})$ est la boule de centre $0$ et de rayon $e^{- 4 \gamma n p}$ dans $\Cc^{k-d_u}$. Cet ensemble est un graphe $(\Phi_{n-1}(Y),Y)$ au-dessus d'une partie de $C_{\gamma}^{-1} (\sigma^{n-1}(\widehat{\beta}))E_s(\sigma^{n-1}(\widehat{\beta}))$ (avec $\Phi_{n-1}(Y)=0$).

On va tirer en arrière ce graphe en utilisant la transformée de graphe dans le cas non inversible du théorème \ref{graphe}.

\begin{lemme}

Il existe un graphe $(\Phi_{n-2}(Y),Y)$ au-dessus de $B(0, e^{- 4 \gamma n p - 2 \gamma }) \subset C_{\gamma}^{-1}(\sigma^{n-2}(\widehat{\beta}))E_s(\sigma^{n-2}(\widehat{\beta}))$ si $\lambda_l \leq 0$ ou au-dessus de $B(0, e^{- 4 \gamma n p - \lambda_l - 2 \gamma }) \subset C_{\gamma}^{-1}(\sigma^{n-2}(\widehat{\beta}))E_s(\sigma^{n-2}(\widehat{\beta}))$ si $\lambda_l >0$ avec $\mbox{Lip } \Phi_{n-2} \leq \xi_0$ et $g_{\sigma^{n-2}(\widehat{\beta})}(\mbox{graphe de } \Phi_{n-2}) \subset \mbox{graphe de } \Phi_{n-1}$.

\end{lemme}

\begin{proof}

Par la proposition \ref{proposition15}, dans le repère

$$C_{\gamma}^{-1}(\sigma^{n-2}(\widehat{\beta})) E_u(\sigma^{n-2}(\widehat{\beta})) \oplus C_{\gamma}^{-1}(\sigma^{n-2}(\widehat{\beta}))E_s(\sigma^{n-2}(\widehat{\beta})),$$

on peut écrire $g_{\sigma^{n-2}(\widehat{\beta})}$ sous la forme

$$g_{\sigma^{n-2}(\widehat{\beta})}(X,Y)=(A_{n-2}X + R_{n-2}(X,Y), B_{n-2}Y + U_{n-2}(X,Y))$$

avec: 

$$Dg_{\sigma^{n-2}(\widehat{\beta})}(0)=A_{\gamma}(\sigma^{n-2}(\widehat{\beta}))=(A_{n-2},B_{n-2})$$

et

\begin{equation*}
\begin{split}
&\max(\|DR_{n-2}(X,Y)\|,\|DU_{n-2}(X,Y)\|) \\
&\leq C dist(f_{n-1}, \mathcal{M})^{-p}  \| C_{\gamma}^{-1}(\sigma^{n-1}(\widehat{\beta}))\|\|C_{\gamma}(\sigma^{n-2}(\widehat{\beta})) \|^2 \|(X,Y)\| \\
\end{split}
\end{equation*}

pour $\|(X,Y)\| \leq \frac{\epsilon_1 dist(f_{n-1}, \mathcal{M})^p}{C \|C_{\gamma}(\sigma^{n-2}(\widehat{\beta}))\|}$ (ici $f_{n-1}=F^{n-2}(f_1)$).

On veut utiliser la transformée de graphe dans le cas non inversible (voir le théorème \ref{graphe}).

Pour cela, plaçons nous dans un premier temps dans le cas où $d_u > 0$. On a $A_{n-2}$ qui est inversible car les exposants associés à $E_u$ ne valent pas $- \infty$. Par ailleurs, par le théorème \ref{pesin} 

$$\|B_{n-2}\| \leq  e^{\gamma} \mbox{  et  } \|A_{n-2}^{-1}\|^{-1} \geq e^{\lambda_{m-1}- \gamma}$$

dans le cas où $\lambda_l \leq 0$ et 

$$\|B_{n-2}\| \leq  e^{\lambda_l + \gamma} \mbox{  et  }  \|A_{n-2}^{-1}\|^{-1} \geq e^{\lambda_{l-1}- \gamma}$$

 si $\lambda_l > 0$. 

En prenant les notations du théorème \ref{graphe} et en utilisant le fait que $\gamma$ est choisi petit par rapport aux différences entre les exposants de Lyapounov et par rapport aux valeurs absolues des exposants non nuls, on en déduit que 

$$1- \xi=\|B_{n-2}\| \|A_{n-2}^{-1}\| \leq e^{- \lambda_{m-1}+ 2 \gamma} \leq e^{-\gamma} < 1$$

dans le premier cas et

$$1- \xi=\|B_{n-2}\| \|A_{n-2}^{-1}\| \leq e^{\lambda_l - \lambda_{l-1}+ 2 \gamma}  \leq e^{-\gamma}  < 1$$

dans le second.

Estimons maintenant le $\delta$ du théorème \ref{graphe}. On a  

$$\| Dg_{\sigma^{n-2}(\widehat{\beta})}(0)-Dg_{\sigma^{n-2}(\widehat{\beta})}(w)\| \leq C dist(f_{n-1}, \mathcal{M})^{-p}  \| C_{\gamma}^{-1}(\sigma^{n-1}(\widehat{\beta}))\|\|C_{\gamma}(\sigma^{n-2}(\widehat{\beta})) \|^2 \| w \| $$

pour $\|w\| \leq \frac{\epsilon_1 dist(f_{n-1}, \mathcal{M})^p}{C \|C_{\gamma}(\sigma^{n-2}(\widehat{\beta}))\|}$. 

Comme $f_1 \in A$, on a par le lemme $9$ de \cite{Det} $dist(f_n, \mathcal{M}) \geq  c(f_1) e^{-\gamma n}$ pour tout $n \geq 1$. De plus, $\|C_{\gamma}^{\pm 1}\|$ étant tempérées, on peut supposer que

$$\| C_{\gamma}^{-1}(\sigma^{n-1}(\widehat{\beta}))\|\|C_{\gamma}(\sigma^{n-2}(\widehat{\beta})) \|^2 \leq \frac{1}{\alpha_0^3} e^{3 \gamma n}.$$

On en déduit que 

$$\| Dg_{\sigma^{n-2}(\widehat{\beta})}(0)-Dg_{\sigma^{n-2}(\widehat{\beta})}(w)\| \leq e^{- \gamma n }$$

si $\| w \| \leq e^{- 2 \gamma n p}$ et $n$ est assez grand car on a supposé que $p \geq 5$.

Pour $n$ assez grand, le $\delta$ du théorème \ref{graphe} est donc aussi petit que l'on veut et on peut appliquer ce théorème avec $R=e^{- 4 \gamma n p}$ et $R_0=e^{- 2 \gamma n p}$ pour obtenir un graphe $(\Phi_{n-2}(Y),Y)$ qui est défini sur 

$$ B \left( 0, \frac{e^{- 4 \gamma n p}}{\max(1, \|B_{n-2}\| + 2 \delta)} \right)  \subset C_{\gamma}^{-1}(\sigma^{n-2}(\widehat{\beta}))E_s(\sigma^{n-2}(\widehat{\beta})) =\Cc^{k - d_u}$$

avec $Lip(\Phi_{n-2}) \leq \xi_0$ et $g_{\sigma^{n-2}(\widehat{\beta})}(\mbox{graphe de } \Phi_{n-2}) \subset \mbox{graphe de } \Phi_{n-1}$.

En utilisant les estimées sur $\|B_{n-2}\|$ précédentes et en prenant $n$ grand par rapport à $\gamma$, on obtient le lemme dans le cas où $d_u > 0$.

Si $d_u = 0$, il n'y a pas de $A_{n-2}$ et on est dans la situation de la remarque qui suit le théorème \ref{graphe}. On prend $R= e^{- 4 \gamma n p}$ et $R_0=e^{- 2 \gamma n p}$ et on a

$$g_{\sigma^{n-2}(\widehat{\beta})} \left( B \left( 0, \frac{e^{- 4 \gamma n p}}{\max(1, \|B_{n-2}\| + 2 \delta)} \right) \right) \subset B(0, e^{- 4 \gamma n p}).$$

Quand $\lambda_1 \leq 0$, on a $\|B_{n-2} \| \leq e^{\gamma}$ et pour $n$ assez grand

$$g_{\sigma^{n-2}(\widehat{\beta})} \left( B \left( 0, e^{- 4 \gamma n p- 2 \gamma} \right) \right) \subset B(0, e^{- 4 \gamma n p})$$

et cela donne le lemme dans ce cas.

Quand $\lambda_1 > 0$, cela signifie que $l=1$ (car $d_u=0$). On a $\|B_{n-2} \| \leq e^{\lambda_l + \gamma}$ et pour $n$ assez grand

$$g_{\sigma^{n-2}(\widehat{\beta})} \left( B \left( 0, e^{- 4 \gamma n p-\lambda_l- 2 \gamma} \right) \right) \subset B(0, e^{- 4 \gamma n p})$$

et cela donne aussi le lemme dans ce dernier cas.

\end{proof}

Maintenant on recommence ce que l'on vient de faire avec $g_{\sigma^{n-3}(\widehat{\beta})}$ à la place de $g_{\sigma^{n-2}(\widehat{\beta})}$. On se place toujours dans une boule $\|w\| \leq e^{-2 \gamma np}$ de sorte à avoir le $\delta$ plus petit que $e^{- \gamma n}$ et on prend $R=e^{- 4 \gamma n p - 2 \gamma }$ ou $R=e^{- 4 \gamma n p -\lambda_l- 2 \gamma }$ suivant le cas où l'on se trouve. On obtient ainsi un graphe $(\Phi_{n-3}(Y),Y)$ au-dessus de $B(0, e^{- 4 \gamma n p - 4 \gamma }) \subset C_{\gamma}^{-1}(\sigma^{n-3}(\widehat{\beta}))E_s(\sigma^{n-3}(\widehat{\beta}))$ si $\lambda_l \leq 0$ ou au-dessus de $B(0, e^{- 4 \gamma n p - 2\lambda_l - 4 \gamma }) \subset C_{\gamma}^{-1}(\sigma^{n-3}(\widehat{\beta}))E_s(\sigma^{n-3}(\widehat{\beta}))$ si $\lambda_l >0$. Ce graphe vérifie $\mbox{Lip } \Phi_{n-3} \leq \xi_0$ et $g_{\sigma^{n-3}(\widehat{\beta})}(\mbox{graphe de } \Phi_{n-3}) \subset \mbox{graphe de } \Phi_{n-2}$. 

On continue ainsi le procédé. A la fin on obtient un graphe $(\Phi_0(Y),Y)$ au-dessus $B(0, e^{- 4 \gamma n p - 2n \gamma }) \subset C_{\gamma}^{-1}(\widehat{\beta})E_s(\widehat{\beta})$ si $\lambda_l \leq 0$ ou au-dessus de $B(0, e^{- 4 \gamma n p - n \lambda_l - 2n \gamma }) \subset C_{\gamma}^{-1}(\widehat{\beta})E_s(\widehat{\beta})$ si $\lambda_l >0$. Ce graphe vérifie $\mbox{Lip } \Phi_{0} \leq \xi_0$ et $g_{\widehat{\beta}}(\mbox{graphe de } \Phi_{0}) \subset \mbox{graphe de } \Phi_{1}$. Son image par $\tau_x \circ C_{\gamma}(\widehat{\beta})$ est la variété stable approchée de $x$ que l'on cherchait. On la notera $W_s(x)$.

Ces variétés stables vérifient la propriété suivante:

\begin{lemme}

Pour $l=0, \cdots , n-1$ on a  

$$f_l \circ \cdots \circ f_1(W_s(x)) \subset B(f_l \circ \cdots \circ f_1(x), e^{- \gamma n}).$$

Ici $f_l \circ \cdots \circ f_1=Id$ pour $l=0$ par convention.

\end{lemme}

\begin{proof}

On a par construction que pour $l=0, \cdots, n-2$

$$g_{\sigma^{l}(\widehat{\beta})}  \circ \cdots \circ g_{\widehat{\beta}}(\mbox{graphe de }  \Phi_0) \subset            
\mbox{graphe de } \Phi_{l+1} \subset B(0, e^{- 2 \gamma n p}).$$

Mais $\sigma^{l}(\widehat{\beta})$ se projette par $\pi$ sur $\tau^{l}(\beta)=(f_{l+1}, f_l \circ \cdots f_1(x))$ d'où

\begin{equation*}
\begin{split}
g_{\sigma^{l}(\widehat{\beta})}  \circ \cdots \circ g_{\widehat{\beta}} &= C_{\gamma}^{-1}(\sigma^{l+1}(\widehat{\beta})) \circ (f_{l+1})_{f_l \circ \cdots \circ f_1(x)} \circ \cdots \circ (f_1)_x \circ C_{\gamma}(\widehat{\beta})\\
&=C_{\gamma}^{-1}(\sigma^{l+1}(\widehat{\beta})) \circ \tau_{f_{l+1} \circ \cdots \circ f_1(x)}^{-1} \circ f_{l+1} \circ \cdots \circ f_1 \circ \tau_x \circ C_{\gamma}(\widehat{\beta}).
\end{split}
\end{equation*}

On a donc

$$C_{\gamma}^{-1}(\sigma^{l+1}(\widehat{\beta})) \circ \tau_{f_{l+1} \circ \cdots \circ f_1(x)}^{-1} \circ f_{l+1} \circ \cdots \circ f_1(W_s(x)) \subset B(0, e^{- 2 \gamma n p})$$

qui donne bien pour $n$ grand

$$f_{l+1} \circ \cdots \circ f_1(W_s(x)) \subset B(f_{l+1} \circ \cdots \circ f_1(x), e^{- \gamma n})$$

grâce au contrôle de $\| C_{\gamma}(\widehat{\beta})\|$ par $1/ \alpha_0$, le fait que $C_{\gamma}$ est tempérée et le contrôle des dérivées premières de $\tau_y$ sur $\Pp^k(\Cc)$.

\end{proof}

\subsection{\bf{Fin de la preuve du théorème d'hyperbolicité}}

Pour chaque $x_i$ ($i=1, \cdots, N$) on a construit un graphe $\Phi_0$ au-dessus d'une partie de $C_{\gamma}^{-1}(\widehat{\beta_i})E_s(\widehat{\beta_i})$. Le volume $2(k-d_u)$-dimensionnel réel de ce graphe est minoré par $e^{2(k-d_u)(- 4 \gamma n p - 2 \gamma n) } \geq e^{- 16k \gamma n p  }$ si $\lambda_l \leq 0$ et par $e^{2(k-d_u)(- 4 \gamma n p - n \lambda_l - 2 \gamma n )} \geq e^{-2(k-d_u)n \lambda_l - 16 k \gamma np }$ si $\lambda_l > 0$. Quand on prend l'image de ce graphe par $C_{\gamma}(\widehat{\beta_i})$, on obtient un graphe au-dessus d'une partie de $E_s(\widehat{\beta_i})$ dans le repère $E_u(\widehat{\beta_i}) \oplus E_s(\widehat{\beta_i})$. Si $(\Phi(Y),Y)$ est l'un d'eux, on a $Lip (\Phi) \leq \frac{\xi_0}{\alpha_0^2}$ qui est aussi petit que l'on veut car on a pris $\xi_0$ petit devant $\alpha_0$. Quitte à remplacer $N$ par $N/K$ où $K$ est une constante qui ne dépend que de $\Pp^k(\Cc)$, on peut supposer que tous ces graphes vivent dans une carte fixée $\psi:U \longrightarrow \Pp^k(\Cc)$ et que les $x_i$ sont à distance au moins $\epsilon_0$ du bord de $U$ (cela siginifie que $\tau_{x_i}$ est égal à $\psi$ modulo une translation). Toujours quitte à remplacer $N$ par $N/K$, on peut supposer que les graphes précédents sont des graphes au-dessus d'un plan complexe $P$ de dimension $k-d_u$ pour la projection orthogonale et que la projection de chaque graphe sur $P$ a un volume supérieur à $e^{- 16k \gamma n p  }$ si $\lambda_l \leq 0$ et supérieur à $e^{-2(k-d_u)n \lambda_l - 16 k \gamma np }$ si $\lambda_l > 0$ (quitte à rediviser par une constante). 

La projection orthogonale de l'union des $N$ variétés $\psi^{-1}(W_s(x_i))$ sur $P$ a donc un volume supérieur à $N e^{- 16k \gamma n p  }$ si $\lambda_l \leq 0$ et supérieur à $ N e^{-2(k-d_u)n \lambda_l - 16 k \gamma np }$ si $\lambda_l > 0$. En particulier, comme la projection orthogonale de $U$ sur $P$ peut être supposée compacte, on peut trouver un plan complexe $L$ de dimension $d_u$ orthogonal à $P$ tel que le nombre d'intersection entre $L$ et $\cup_{i=1}^N \psi^{-1}(W_s(x_i))$ est supérieur à $N'=N e^{- 16k \gamma n p  }$ si $\lambda_l \leq 0$ et supérieur à $ N'=N e^{-2(k-d_u)n \lambda_l - 16 k \gamma np }$ si $\lambda_l > 0$ (éventuellement divisé par une contante). Remarquons juste que dans le cas où $d_u=0$, ce plan complexe n'est rien d'autre qu'un point $L$ qui est dans un nombre supérieur à $N'=N e^{- 16k \gamma n p  }$ de $\psi^{-1}(W_s(x_i))$ si $\lambda_l \leq 0$ ou dans un nombre supérieur à $ N'=N e^{-2(k-d_u)n \lambda_l - 16 k \gamma np }$ de $\psi^{-1}(W_s(x_i))$ si $\lambda_l > 0$.

Replaçons-nous dans le cas général. L'intersection entre chaque $\psi^{-1}(W_s(x_i))$ et $L$ se fait en au plus un point car ce sont des graphes au-dessus de $P$. Notons $y_1, \cdots , y_{N'}$ ces points d'intersections et pour d'alléger les notations, supposons qu'ils correspondent à $x_1, \cdots, x_{N'}$. 

Nous allons majorer $N'$ en utilisant un argument d'entropie et de volume associé à $L$. Nous obtiendrons ainsi une contradiction dans le cas où $\lambda_l \leq 0$ et la minoration de $\lambda_l $ cherchée dans l'autre cas.

Le plan complexe $L$ dans la carte correspond si on veut à un plan projectif $R$ dans $\Pp^k(\Cc)$ qui contient $\psi(L \cap U)$. Le point crucial est le suivant:

\begin{lemme}

Les points $\psi(y_1), \cdots , \psi(y_{N'})$ sont des points de $R$ qui sont $(n, \frac{\epsilon}{2}, f_1)$-séparés si $n$ est assez grand par rapport à $\epsilon$.

\end{lemme}

\begin{proof}

Soient $i,j \in \{1, \cdots , N' \}$ avec $i \neq j$. Par construction, chaque point $\psi(y_i)$ se trouve dans $R \cap W_s(x_i)$ et les points $x_i$ sont $(n, \epsilon, f_1)$ séparés. Il existe donc $0 \leq l \leq n-1$ avec $dist(f_l \circ \cdots \circ f_1(x_i) , f_l \circ \cdots \circ f_1(x_j) ) \geq \epsilon$.

Maintenant, par le lemme précédent 

$$f_l \circ \cdots \circ f_1(W_s(x_i)) \subset B(f_l \circ \cdots \circ f_1(x_i), e^{- \gamma n})$$

d'où $dist(f_l \circ \cdots \circ f_1(x_i), f_l \circ \cdots \circ f_1(\psi(y_i))) \leq e^{- \gamma n} < \frac{\epsilon}{4}$ pour $n$ assez grand et on a la même chose en remplaçant $i$ par $j$. 

En utilisant que la distance entre $f_l \circ \cdots \circ f_1(x_i)$ et $f_l \circ \cdots \circ f_1(x_j)$ est plus grande que $\epsilon$, on obtient donc

$$dist(f_l \circ \cdots \circ f_1(\psi(y_i)), f_l \circ \cdots \circ f_1(\psi(y_j))) > \frac{\epsilon}{2}.$$

C'est ce que l'on voulait démontrer.

\end{proof}

Remarquons qu'avec ce lemme nous avons déjà le théorème d'hyperbolicité dans le cas où $d_u=0$. En effet, comme tous les $\psi(y_i)$ sont égaux au point $\psi(L)$, ils ne peuvent pas être $(n, \frac{\epsilon}{2}, f_1)$-séparés. On a donc $N' \leq 1$. Comme $N \geq \frac{1}{2} d^{kn} e^{- \gamma n}$, on obtient ainsi une contradiction dans le cas où $\lambda_l \leq 0$ et lorsque $\lambda_l >0$ (i.e. $l=1$ car $d_u=0$), on a

$$\frac{1}{2} d^{kn} e^{ - \gamma n} e^{-2kn \lambda_l - 16 k \gamma np } \leq N' \leq 1$$

ce qui donne bien $\lambda_l \geq \frac{\log d}{2}$.

On suppose donc dans la suite que $d_u >0$ et il reste à majorer le cardinal d'un ensemble $(n, \frac{\epsilon}{2}, f_1)$-séparés dans $R$ pour obtenir la majoration de $N'$. Pour cela, on va utiliser la méthode de Gromov (voir \cite{Gr}) comme dans la preuve du théorème \ref{Gromov}.

Pour $i=1, \cdots, N'$, les $\psi(y_i) \in R$ donnent des points

$$Y_i=(\psi(y_i), f_1(\psi(y_i)), \cdots, f_{n-1} \circ \cdots \circ  f_1(\psi(y_i)))$$

qui sont $\frac{\epsilon}{2}$-séparés dans $(\Pp^k(\Cc))^n$. Les boules $B(Y_i, \frac{\epsilon}{4})$ sont donc disjointes et si on note

$$\Gamma_n:=\{(x, f_1(x), \cdots , f_{n-1} \circ \cdots \circ f_1(x)), x\in R \},$$

le théorème de Lelong implique que le volume de $\Gamma_n$ intersecté avec une boule $B(Y_i, \frac{\epsilon}{4})$ est minoré par $C(d_u) \left( \frac{\epsilon}{4} \right)^{2 d_u}$. Cela donne une minoration du volume de $\Gamma_n$ par $C(d_u) \left( \frac{\epsilon}{4} \right)^{2 d_u} N'$. 

On va maintenant majorer ce volume comme dans la preuve du théorème \ref{Gromov}, en utilisant la cohomologie des $f_i$.

Soit $\omega_n=\sum_{i=1}^n \pi_i^* \omega$ où $\pi_i$ ($i=1, \cdots , n$) est la projection de $(\Pp^k(\Cc))^n$ sur sa $i$-ème coordonnée et $\omega$ est la forme de Fubini-Study de $\Pp^k(\Cc)$. Le volume de $\Gamma_n$ est égal à

\begin{equation*}
\begin{split}
\int_{\Gamma_n} \omega_n^{d_u} &= \sum_{1 \leq n_1 , \cdots , n_{d_u} \leq n} \int_{\Gamma_n} \pi_{n_1}^* \omega \wedge \cdots \wedge \pi_{n_{d_u}}^* \omega \\
&=\sum_{1 \leq n_1 , \cdots , n_{d_u} \leq n} \int_{R} (f_{n_1 -1} \circ \cdots \circ f_1)^* \omega \wedge \cdots \wedge (f_{n_{d_u} -1} \circ \cdots \circ f_1)^* \omega.\\
\end{split}
\end{equation*}

Comme $f_1$ est dans $A$, les $f_i$ sont des endomorphismes holomorphes de degré $d$. En particulier, $(f_{n_l -1} \circ \cdots \circ f_1)^* \omega$ est cohomologue à $d^{n_l -1} \omega$ et alors

$$\int_{\Gamma_n} \omega_n^{d_u} =\sum_{1 \leq n_1 , \cdots , n_{d_u} \leq n} d^{n_1 + \cdots + n_{d_u}}d^{-d_u} \leq n^{d_u} d^{-d_u} d^{nd_u}.$$

On obtient ainsi

$$N' \leq \frac{n^{d_u} d^{-d_u} d^{nd_u}}{C(d_u) \left( \frac{\epsilon}{4} \right)^{2d_u}}.$$

Si on fait le bilan, dans le cas où $\lambda_l \leq 0$ on a pour $n$ grand

$$\frac{1}{2} d^{kn}  e^{- \gamma n} e^{-16 k \gamma n p} \leq N e^{-16 k \gamma n p}=N' \leq \frac{n^{d_u} d^{-d_u} d^{nd_u}}{C(d_u) \left( \frac{\epsilon}{4} \right)^{2d_u}}$$

qui est absurde car $d_u < k$ et dans le cas où $\lambda_l >0$, on a pour $n$ grand

$$\frac{1}{2} d^{kn} e^{ - \gamma n} e^{-2(k-d_u)n\lambda_l-16 k \gamma n p} \leq N e^{-2(k-d_u)n\lambda_l-16 k \gamma n p}=N' \leq \frac{n^{d_u} d^{-d_u} d^{nd_u}}{C(d_u) \left( \frac{\epsilon}{4} \right)^{2d_u}}$$

ce qui donne bien que $\lambda_l \geq \frac{\log d}{2}$.

\bigskip

\bigskip\noindent
Henry De Thélin, Université Paris 13, Sorbonne Paris Cité, LAGA, CNRS (UMR 7539), F-93430, Villetaneuse, France.\\
 {\tt dethelin@math.univ-paris13.fr}

\end{document}